\documentclass[11pt,a4paper]{article}

\usepackage[T1]{fontenc}
\usepackage{lmodern}
\usepackage[margin=1in]{geometry}
\usepackage{microtype}
\usepackage{amsmath,amssymb,amsthm,mathtools}
\usepackage{enumerate}
\usepackage{authblk}
\usepackage[hidelinks,hypertexnames=false]{hyperref}
\hypersetup{
  pdftitle={Cyclic Shuffle Groups: Universal Two-Transitivity and Complete Classification},
  pdfauthor={Benjamin Marsh},
  pdfsubject={Finite permutation groups and generalized shuffle groups},
  pdfkeywords={cyclic shuffle groups, generalized shuffle groups, finite permutation groups, two-transitive groups, fixed point ratios}
}

\newcommand{\Alt}{\operatorname{Alt}}
\newcommand{\Sym}{\operatorname{Sym}}
\newcommand{\Sh}{\operatorname{Sh}}
\newcommand{\AGL}{\operatorname{AGL}}
\newcommand{\GL}{\operatorname{GL}}
\newcommand{\PSL}{\operatorname{PSL}}
\newcommand{\PGL}{\operatorname{PGL}}
\newcommand{\PSU}{\operatorname{PSU}}

\newcommand{\PGammaL}{\operatorname{P\Gamma L}}
\newcommand{\PGammaU}{\operatorname{P\Gamma U}}
\newcommand{\Soc}{\operatorname{Soc}}
\newcommand{\Sz}{\operatorname{Sz}}
\newcommand{\fpr}{\operatorname{fpr}}
\newcommand{\Fix}{\operatorname{Fix}}
\newcommand{\sgn}{\operatorname{sgn}}

\newtheorem{theorem}{Theorem}[section]
\newtheorem{proposition}[theorem]{Proposition}
\newtheorem{lemma}[theorem]{Lemma}
\newtheorem{corollary}[theorem]{Corollary}

\theoremstyle{definition}
\newtheorem{definition}[theorem]{Definition}

\theoremstyle{remark}

\title{Cyclic Shuffle Groups:\\ Universal Two-Transitivity and Complete Classification}
\author{Benjamin Marsh}
\affil{Sei Labs and University of Portsmouth\\\small ben@seinetwork.io}
\date{August 2026}

\begin{document}

\maketitle

\begin{abstract}
Let \(k\geq 3\), \(n\geq 1\), and let \(H_{k,n}=\Sh(C_k,n)\) be the group generated by the standard \(k\) pile perfect shuffle and cyclic pile permutation on a deck of \(kn\) cards. We prove that \(H_{k,n}\) is \(2\)-transitive whenever \(n\) is not a power of \(k\). Residual commutators give translations supported on two pile labels, and a strongly connected digit digraph propagates these translations throughout the deck, a separate argument resolves the antipodal support case. We then combine this result with fixed point ratio bounds for primitive groups and explicit boundary calculations to determine \(H_{k,n}\) for all \(k\) and \(n\). If \(n=k^f\), then \(H_{k,n}\cong C_k\wr C_{f+1}\). If \(k=4\) and \(n=2\cdot4^j\), then \(H_{4,n}\cong\AGL(2j+3,2)\). In every other case, \(H_{k,n}\) is \(\Alt(kn)\) or \(\Sym(kn)\), according to the parity of its generators. This proves Conjecture~1.10 of Amarra, Morgan and Praeger and Conjecture~5.1 of Xia, Zhang, Zhang and Zhu. More generally, we classify \(\Sh(P,n)\) for every pile group \(P\) containing \(C_k\), and obtain the odd \(k\) part of their Conjecture~5.2.
\end{abstract}



\section{Introduction and statement of the result}
\label{sec:introduction}

Put
\[
  \Omega_{k,n}=\{0,1,\ldots,kn-1\}.
\]
Every \(x\in\Omega_{k,n}\) has a unique pile coordinate expression \(x=an+b\), where \(0\leq a<k\) and \(0\leq b<n\). Define
\begin{align}
  \sigma(an+b)&=kb+a,\label{eq:shuffle}\\
  \rho(an+b)&=((a+1)\bmod k)n+b.\label{eq:pile-cycle}
\end{align}
Thus \(\sigma\) is the standard perfect shuffle and \(\rho\) cyclically permutes the \(k\) piles. Throughout, products are composed right to left, \((gh)(x)=g(h(x))\). We study
\begin{equation}
  H_{k,n}=\langle\sigma,\rho\rangle.
\label{eq:cyclic-shuffle-group}
\end{equation}
Write
\[
  C_k=\langle(0\ 1\ \cdots\ k-1)\rangle\leq\Sym(k)
\]
for the standard cyclic pile group. Equivalently,
\[
  H_{k,n}=\langle\sigma,\sigma\rho\rangle,\qquad
  (\sigma\rho)(an+b)=kb+(a+1\bmod k),
\]
so this is the cyclic generalized shuffle group \(\Sh(C_k,n)\) in the notation of \cite{AMP,XZZZ}. The group theoretic study of perfect shuffles includes the two pile analysis of Diaconis, Graham and Kantor \cite{DGK} and the general pile treatment of Medvedoff and Morrison \cite{MM}; generalized shuffle groups are developed systematically in \cite{AMP}. Amarra, Morgan and Praeger formulated the cyclic alternating containment problem as \cite[Conjecture~1.10]{AMP}. Xia, Zhang, Zhang and Zhu classified the full pile group
\[
  G_{k,kn}=\Sh(\Sym(k),n),
\]
in their Theorem~1.5. Their Section~5 treats \(\Sh(C_k,n)\) as an open problem, describes its determination as the ``best possible'' strengthening of their theorem, and formulates the exact cyclic classification as \cite[Conjecture~5.1, pp.~14--15]{XZZZ}. To the author's knowledge, neither this conjecture nor the related alternating pile conjecture has subsequently been resolved. Theorem~\ref{thm:classification} proves Conjecture~5.1, and hence Conjecture~1.10 of \cite{AMP}. Theorem~\ref{thm:cycle-containing-pile-groups} further classifies \(\Sh(P,n)\) whenever \(C_k\leq P\), from which the odd \(k\) part of \cite[Conjecture~5.2]{XZZZ} follows. No claim is made here for the even \(k\) part of Conjecture~5.2, for Conjecture~5.3, or for Question~5.4 of \cite{XZZZ}.

\begin{theorem}[]\label{thm:classification}
Let \(k\geq 3\) and \(n\geq 1\).
\begin{enumerate}[(i)]
\item If \(n=k^f\) for some \(f\geq 0\), then, in product action,
\begin{equation}
  H_{k,n}\cong C_k^{\,f+1}\rtimes C_{f+1}
  =C_k\wr C_{f+1},
\label{eq:power-deck-classification}
\end{equation}
where \(C_{f+1}\) cyclically permutes the \(f+1\) factors of \(C_k^{\,f+1}\).

\item If \(k=4\) and \(n=2\cdot 4^j\) for some \(j\geq 0\), then
\begin{equation}
  H_{4,n}\cong\AGL(2j+3,2)
\label{eq:four-pile-classification}
\end{equation}
in its natural \(2\)-transitive action of degree \(4n=2^{2j+3}\).

\item In all remaining cases,
\begin{equation}
H_{k,n}=
\begin{cases}
\Alt(kn),&
n(k-1)\equiv 0\pmod 2\ \text{and}\
\displaystyle\binom{k}{2}\binom{n}{2}\equiv 0\pmod 2,\\
\Sym(kn),&\text{otherwise}.
\end{cases}
\label{eq:generic-classification}
\end{equation}
\end{enumerate}
\end{theorem}

The power deck identification in Theorem~\ref{thm:classification}(i) is known; see \cite[Theorem~1.4(1)]{AMP}. The related four pile affine calculation goes back to Cohen, Harmse, Morrison and Wright \cite{CHMW}, while the cyclic identity in part~(ii) is recorded in \cite[(12), p.~14]{XZZZ}. For completeness, Section~\ref{sec:power-parity} gives a direct binary coordinate proof that makes the exceptional action explicit. The proof of part~(iii) is the main content of this paper. Its elementary core is the following uniform theorem.

\begin{theorem}[]\label{thm:two-transitive}
Let \(k\geq 3\). If \(n\) is not a power of \(k\), then \(H_{k,n}\) is \(2\)-transitive on \(\Omega_{k,n}\).
\end{theorem}

Write
\begin{equation}
  n=k^s t,
\label{eq:residual-decomposition}
\end{equation}
where \(s\geq 0\) is maximal, thus \(k\nmid t\). Put
\begin{equation}
  \lambda(k,t)=\frac{k}{\gcd(k,t)}.
\label{eq:residual-cycle-length}
\end{equation}

\begin{theorem}[]\label{thm:residual-cycle}
Let \(k\geq 3\), write \(n=k^s t\), where \(s\geq0\) is maximal (equivalently, \(k\nmid t\)), and suppose that \(t>1\). If
\[
  \lambda(k,t)\geq 3,
\]
then \(H_{k,n}\) is \(2\)-transitive.
\end{theorem}

For odd \(k\), the integer \(\lambda(k,t)>1\) is odd, so Theorem~\ref{thm:residual-cycle} already proves Theorem~\ref{thm:two-transitive}. The only remaining case has \(k=2h\) and \(\gcd(k,t)=h\). A separate argument handles this antipodal support configuration for every \(h\geq 2\), the boundary value \(h=2\) requires a short parity argument on the residual coordinate. The passage from \(2\)-transitivity to alternating group containment separates naturally by pile number. For \(k\geq 5\), a prime-order power of the controlled commutator has fixed point ratio \((k-2)/k\geq 3/5>1/\sqrt{3}\), so Corollary 2 of Burness--Guralnick applies even when the selected prime is \(2\). For \(k=4\), an order \(3\) specialization of their almost simple result handles all nonaffine cases. The three pile case lies below the product action threshold, there we use an explicit large degree classified reduction and three displayed long prime orbits.

\section{Power decks and parity}
\label{sec:power-parity}

We first record the power deck and four pile affine identifications, followed by the parity calculation.

\begin{proposition}[]\label{prop:power-decks}
Let \(k\geq 2\) and let \(n=k^f\). Then
\[
  H_{k,n}\cong C_k^{\,f+1}\rtimes C_{f+1}
  =C_k\wr C_{f+1}.
\]
\end{proposition}

\begin{proof}
Identify the deck with the set of base \(k\) words
\[
  (x_f,x_{f-1},\ldots,x_0)\in(\mathbb Z/k\mathbb Z)^{f+1}.
\]
The shuffle \(\sigma\) cyclically rotates the coordinates, while \(\rho\)
adds \(1\) to the first coordinate. Hence
\[
  \rho_i=\sigma^{-i}\rho\sigma^i,\qquad 0\leq i\leq f,
\]
adds \(1\) independently in the \(i\)th coordinate. The elements \(\rho_i\)
generate a regular translation subgroup \(T\cong C_k^{\,f+1}\). The element \(\sigma\) normalizes \(T\) by cyclically permuting its factors. Every power of \(\sigma\) fixes the zero word, whereas no nonidentity element of \(T\) does, so \(T\cap\langle\sigma\rangle=1\). Therefore
\[
  H_{k,k^f}=T\rtimes\langle\sigma\rangle
  \cong C_k^{\,f+1}\rtimes C_{f+1}.
\]
\end{proof}

\begin{proposition}[]
\label{prop:four-pile-affine}
If \(n=2\cdot4^j\) with \(j\geq0\), then
\[
  H_{4,n}\cong\AGL(2j+3,2)
\]
in its natural action.
\end{proposition}

\begin{proof}
Put \(m=2j+3\), so \(4n=2^m\), and identify a position with
\[
  (x_0,\ldots,x_{m-1})\in V:=\mathbb F_2^m,
  \qquad
  x=\sum_{i=0}^{m-1}2^i x_i.
\]
Thus the pile label is \(a=x_{m-2}+2x_{m-1}\).  If \(S\) denotes the linear map induced by \(\sigma\), formula~\eqref{eq:shuffle} gives
\[
  S(e_i)=e_{i+2\pmod m}.
\]
Because \(m\) is odd, this is a full cyclic permutation of the coordinate vectors. Write \(t_v\) for translation by \(v\), and use the convention that \(E_{ij}\) sends \(e_j\) to \(e_i\).  Addition of one modulo \(4\) on the two pile bits shows that
\[
  \rho=t_{e_{m-2}}L,
  \qquad
  L=I+E_{m-1,m-2}.
\]
In particular,
\[
  \rho^2=t_{e_{m-1}}.
\]
Conjugating by powers of \(S\) therefore gives translation by every coordinate vector, so \(H_{4,n}\) contains the full translation group \(T=\{t_v:v\in V\}\). Cancelling the translation part of \(\rho\) now gives \(t_{e_{m-2}}\rho=L\in H_{4,n}\).  The \(S\)-conjugates of \(L\) are all cyclically adjacent transvections
\[
  I+E_{r+1,r}\qquad(r\in\mathbb Z/m\mathbb Z).
\]
Using
\[
  [I+E_{i,j},I+E_{j,k}]=I+E_{i,k}
  \qquad(i,j,k\ \text{distinct})
\]
along directed paths around the coordinate cycle supplies \(I+E_{i,k}\) for every \(i\neq k\). These elementary transvections generate \(\GL(m,2)\). Hence \(\AGL(m,2)\leq H_{4,n}\). Conversely, \(\sigma\) is linear and \(\rho\) is affine, so \(H_{4,n}\leq\AGL(m,2)\), and equality follows.
\end{proof}

\begin{proposition}[]\label{prop:signs}
For all \(k,n\geq 1\),
\begin{align}
  \sgn(\sigma)&=(-1)^{\binom{k}{2}\binom{n}{2}},\label{eq:shuffle-sign}\\
  \sgn(\rho)&=(-1)^{n(k-1)}.\label{eq:pile-cycle-sign}
\end{align}
Consequently, if \(\Alt(kn)\leq H_{k,n}\), then \eqref{eq:generic-classification} holds.
\end{proposition}

\begin{proof}
The map \(\sigma\) changes the ordering of a \(k\times n\) array from pile first order to depth first order. An inversion is exactly a pair
\[
  a<a',\qquad b>b',
\]
so the number of inversions is \(\binom{k}{2}\binom{n}{2}\). This proves \eqref{eq:shuffle-sign}. The permutation \(\rho\) is a product of \(n\) disjoint \(k\)-cycles, which proves \eqref{eq:pile-cycle-sign}. The final assertion is immediate.
\end{proof}

\section{Residual coordinates and a controlled commutator}
\label{sec:residual-coordinates}

Assume henceforth that \(n=k^s t\), where \(s\geq0\) is maximal, \(t>1\), and \(k\nmid t\), and put
\begin{equation}
  d=s+1.
\label{eq:digit-depth}
\end{equation}
Every position has a unique expression
\begin{equation}
  x=At+X,\qquad 0\leq A<k^d,\qquad 0\leq X<t.
\label{eq:digit-residual-coordinate}
\end{equation}
Write the base-\(k\) expansion of \(A\) as
\[
  A=(a_0,a_1,\ldots,a_{d-1})_k
   =\sum_{i=0}^{d-1}a_i k^{d-1-i}.
\]
The first digit \(a_0\) is the pile label. We often write a point as \((a,B,X)\), where \(a=a_0\) and \(B=(a_1,\ldots,a_{d-1})\). For \(0\leq i\leq d\), define
\begin{equation}
  u_i=\sigma^{-i}\rho\sigma^i.
\label{eq:digit-conjugates}
\end{equation}
We first verify the action of the conjugates \(u_i\) for \(i<d\). Write
\[
  A=a_0k^{d-1}+A',
\]
and write
\[
  kX+a_0=qt+R,\qquad 0\leq q<k,\qquad 0\leq R<t.
\]
The definition of \(\sigma\) gives
\[
  \sigma(a_0,a_1,\ldots,a_{d-1},X)
  =(a_1,\ldots,a_{d-1},q,R).
\]
The element \(u_0=\rho\) increments \(a_0\) and fixes all other digits and \(X\). Suppose that \(0\leq i<d-1\) and that \(u_i\) increments the \(i\)th displayed digit in digit residual coordinates. Since
\[
  u_{i+1}=\sigma^{-1}u_i\sigma,
\]
the element \(u_i\), after the first application of \(\sigma\), increments the displayed digit \(a_{i+1}\) and leaves \(q\), \(R\), and all other displayed digits unchanged. The inverse shuffle reconstructs \(a_0\) and \(X\) uniquely from
\[
  kX+a_0=qt+R,
\]
so it leaves them unchanged. It follows by induction that, for every \(i<d\), the element \(u_i\) increments \(a_i\) modulo \(k\) and fixes all other digits and the residual coordinate \(X\). The last conjugate \(u_d\) no longer acts on a stored base \(k\) digit and creates the key residual motion.

\begin{lemma}[]\label{lem:shift-register}
Write
\begin{equation}
  t=km+r,\qquad 1\leq r\leq k-1,
\label{eq:divide-t-by-k}
\end{equation}
and set
\[
  \varepsilon_a=\left\lfloor\frac{a+r}{k}\right\rfloor
  \qquad(0\leq a<k).
\]
Then
\begin{equation}
u_d(a,B,X)=
\bigl(a+r\bmod k,\ B,\ X+m+\varepsilon_a\bmod t\bigr).
\label{eq:ud-action}
\end{equation}
\end{lemma}

\begin{proof}
Under one application of \(\sigma\), write
\begin{equation}
  kX+a=qt+R,\qquad 0\leq q<k,\qquad 0\leq R<t.
\end{equation}
Then
\begin{equation}
  \sigma(a,B,X)=(B,q,R).
\end{equation}
Since \(u_{d-1}\) increments the final displayed base \(k\) digit \(q\), the conjugate \(u_d=\sigma^{-1}u_{d-1}\sigma\) replaces the residue class of \(kX+a\) modulo \(kt\) by that of \(kX+a+t\). Now
\[
  kX+a+t=k(X+m+\varepsilon_a)+(a+r-k\varepsilon_a).
\]
If this integer is at least \(kt\), subtracting \(kt\) changes the quotient by \(-t\), which disappears modulo \(t\). This gives \eqref{eq:ud-action}.
\end{proof}

We use the commutator convention
\[
  [g,h]=g^{-1}h^{-1}gh.
\]

\begin{proposition}[]\label{prop:gate}
Let
\begin{equation}
  c=[\rho,u_d].
\label{eq:controlled-commutator}
\end{equation}
Then
\begin{equation}
  c(a,B,X)=
  \bigl(a,B,X+\varepsilon_a-\varepsilon_{a+1}\bmod t\bigr),
\label{eq:commutator-action}
\end{equation}
where \(a+1\) is read modulo \(k\). In particular, \(c\) is the identity on all piles except
\[
  k-r-1\qquad\text{and}\qquad k-1,
\]
on which it translates \(X\) by \(-1\) and \(+1\), respectively.
\end{proposition}

\begin{proof}
Applying the four factors of \(\rho^{-1}u_d^{-1}\rho u_d\) gives
\[
\begin{aligned}
(a,B,X)&\xmapsto{\ u_d\ }(a+r,B,X+m+\varepsilon_a)\\
&\xmapsto{\ \rho\ }(a+r+1,B,X+m+\varepsilon_a)\\
&\xmapsto{\ u_d^{-1}\ }(a+1,B,X+\varepsilon_a-\varepsilon_{a+1})\\
&\xmapsto{\ \rho^{-1}\ }(a,B,X+\varepsilon_a-\varepsilon_{a+1}),
\end{aligned}
\]
with pile labels reduced modulo \(k\) and residual coordinates modulo \(t\). Since \(\varepsilon_a=0\) for \(0\leq a\leq k-r-1\) and \(\varepsilon_a=1\) for \(k-r\leq a\leq k-1\), the difference is nonzero at exactly the two stated piles.
\end{proof}

The supports of the conjugates \(\rho^{-j}c\rho^j\) are the edges of the circulant graph
\begin{equation}
  C(k,r):\qquad a\sim a+r\qquad(a\in\mathbb Z/k\mathbb Z).
\label{eq:support-graph}
\end{equation}
Put
\begin{equation}
  g=\gcd(k,r)=\gcd(k,t),\qquad \lambda=\frac{k}{g}.
\label{eq:support-components}
\end{equation}
Then \(C(k,r)\) is a disjoint union of \(g\) cycles of length \(\lambda\).

\section{A universal digit digraph}
\label{sec:digit-digraph}

For \(0\leq A<k^d\), let
\[
  I_A=\{At,At+1,\ldots,At+t-1\}
\]
be the residual block with base \(k\) word \(A\). Call \(I_A\) \emph{active} if its leading digit is nonzero. Suppose an element of the point stabilizer \((H_{k,n})_0\) translates the residual coordinate by \(1\) or \(-1\) throughout every block in a fixed pile. As soon as an orbit meets such a block, it contains the entire block. The following digit digraph controls how filled blocks propagate under \(\sigma^{-1}\).

\begin{definition}
Let \(D_{k,t}\) be the directed graph with vertex set
\[
  V_k=\{1,2,\ldots,k-1\}\subset\mathbb Z/k\mathbb Z
\]
and an arc \(z\to a\) whenever
\begin{equation}
  a\equiv tz+j\pmod k
  \quad\text{for some }0\leq j<t,\qquad a\neq 0.
\end{equation}
\end{definition}

\begin{lemma}[]\label{lem:digit-connectivity}
For every \(k\geq 2\) and \(t\geq 2\), the digraph \(D_{k,t}\) is strongly connected. Moreover, the vertex \(k-1\) has a loop. If
\[
  M=\lceil\log_t k\rceil,
\]
then its directed diameter is at most \(3M-1\).
\end{lemma}

\begin{proof}
We first show that \(k-1\) reaches every vertex. Choose \(M\) minimally so that \(t^M\geq k\). Given \(y\in V_k\), put
\[
  z=t^M-k+y,\qquad 0\leq z<t^M,
\]
and write \(z\) as an \(M\) digit base \(t\) word, allowing leading zeros. Starting at \(k-1\equiv -1\pmod k\) and reading the first \(i\) digits gives a state congruent to
\[
  z_i-t^i\pmod k,
\]
where \(0\leq z_i<t^i\) is the corresponding prefix. For \(i<M\), minimality of \(M\) gives \(t^i<k\), and hence
\[
  -k<z_i-t^i<0.
\]
Thus no intermediate state is \(0\). At \(i=M\), the state is \(z-t^M\equiv y\pmod k\). Therefore \(k-1\) reaches \(y\) inside \(D_{k,t}\). Conversely, fix \(x\in V_k\). For the same minimal \(M\), after reading an \(M\) digit word with numerical value \(z\in[0,t^M-1]\), the resulting state in the full generalized de Bruijn digraph is
\[
  t^M x+z\pmod k.
\]
Since an interval of \(t^M\) consecutive integers contains a complete residue system modulo \(k\), some word takes \(x\) to \(k-1\). If its path avoids \(0\), it is a path in \(D_{k,t}\). Otherwise, consider the first step into \(0\). Let \(p\neq 0\) be the preceding state and let the digit used be \(j\), so \(tp+j\equiv 0\pmod k\). If \(j>0\), replace \(j\) by \(j-1\) and stop. The new state is \(k-1\). If \(j=0\), write the ordinary base \(t\) expansion of \(k-1\) as \(a_1a_2\cdots a_e\), with \(a_1\neq 0\). Replace the digit \(0\) by \(a_1\) and then read \(a_2,\ldots,a_e\). Since \(tp\equiv 0\pmod k\), the successive states are the positive integer prefixes of the base \(t\) expansion of \(k-1\). They all lie in \(V_k\), and the final state is \(k-1\). Hence every vertex reaches \(k-1\), proving strong connectivity. Finally,
\[
  t(k-1)+(t-1)\equiv k-1\pmod k,
\]
so \(k-1\) has a loop. The first construction gives a path of length \(M\) from \(k-1\) to every vertex. In the reverse direction, a path avoiding \(0\) has length \(M\). If the first attempted step into \(0\) is repaired as above, its length is at most \((M-1)+M=2M-1\). Hence any ordered pair of vertices can be joined via \(k-1\) by a path of length at most
\[
  (2M-1)+M=3M-1.
\]
\end{proof}

We now translate this graph statement into a permutation group statement.

\begin{lemma}[]\label{lem:backward-block}
Let a block word be \(w_0w_1\cdots w_{d-1}\). Under \(\sigma^{-1}\) it can move to
\begin{equation}
  aw_0w_1\cdots w_{d-2},\qquad
  a\equiv tw_{d-1}+X\pmod k,
\label{eq:backward-block-move}
\end{equation}
where \(X\) is the residual coordinate chosen in the source block. Consequently, at the filled block level, where a residual gate is used after each move to reselect \(X\) in the newly filled active block, one complete sequence of \(d\) backward moves can replace each digit independently by an out neighbor in \(D_{k,t}\).
\end{lemma}

\begin{proof}
For \(y\in\Omega_{k,n}\),
\begin{equation}
  \sigma^{-1}(y)=(y\bmod k)n+\left\lfloor\frac{y}{k}\right\rfloor.
\label{eq:inverse-shuffle}
\end{equation}
Write \(A=kC+z\) and \(y=At+X\). Then
\[
  \left\lfloor\frac{At+X}{k}\right\rfloor
  =Ct+\left\lfloor\frac{zt+X}{k}\right\rfloor,
\]
and the final quotient lies in \([0,t-1]\). The new pile label is \(y\bmod k\equiv tz+X\pmod k\), which proves \eqref{eq:backward-block-move}. Iterating \(d\) times processes each digit once and restores the original digit order.
\end{proof}

\begin{proposition}[]\label{prop:full-gate}
Assume that for every pile label \(a\in\{1,\ldots,k-1\}\), the stabilizer \((H_{k,n})_0\) contains an element that acts on every residual block in pile \(a\) as \(X\mapsto X+1\) or \(X\mapsto X-1\), with a fixed choice of sign for that gate. Then \(H_{k,n}\) is \(2\)-transitive.
\end{proposition}

\begin{proof}
Let \(L\leq (H_{k,n})_0\) be generated by \(\sigma\) and one such residual gate for each nonzero pile, and put \(\mathcal O=1^L\). If \(\mathcal O\) meets an active block, then the relevant gate fills the entire block. Since
\[
  \sigma^{-1}(1)=n=k^{d-1}t,
\]
the orbit \(\mathcal O\) fills the block with word
\begin{equation}
  10^{d-1}.
\end{equation}
Applying \(d-1\) backward moves and choosing \(X=1\) whenever the outgoing digit is \(0\) produces the filled all nonzero block \(1^d\). For each \(z\in V_k\), strong connectivity in Lemma~\ref{lem:digit-connectivity} supplies a directed path \(P_z\) from \(1\) to \(z\) that passes through \(k-1\). Let \(\ell_z\) be its length and put
\[
  E=\max_{z\in V_k}\ell_z.
\]
Since \(k-1\) has a loop, inserting \(E-\ell_z\) copies of this loop at the visit to \(k-1\) extends every \(P_z\) to a path of the common length \(E\). Now fix an all nonzero target word \(z_0z_1\cdots z_{d-1}\), and assign \(P_{z_i}\) to its \(i\)th coordinate. Starting from the filled block \(1^d\), perform \(E\) complete backward sweeps. By Lemma~\ref{lem:backward-block}, the \(j\)th sweep replaces every coordinate digit by the next vertex on its assigned path and restores the digit order after \(d\) moves. At every individual move, the new leading digit is a vertex of \(D_{k,t}\), hence is nonzero. The block encountered is therefore active and can be filled before the next move. After \(E\) sweeps the target block is filled. Thus every all nonzero block is filled. It remains to fill an arbitrary active block with word
\[
  z_0z_1\cdots z_{d-1},\qquad z_0\neq 0.
\]
A forward application of \(\sigma\) drops the leading digit \(a\) and appends the quotient \(q\) in
\begin{equation}
  kX+a=qt+R,\qquad 0\leq R<t.
\end{equation}
For each prescribed \(q\in\{0,\ldots,k-1\}\), the interval \([qt,(q+1)t-1]\) contains two consecutive integers and hence an integer not divisible by \(k\). Write such an integer as \(kX+a\) with \(1\leq a\leq k-1\) and \(0\leq X<t\). Thus every desired quotient can be appended from some nonzero leading digit. Choose an all nonzero source word \(a_0\cdots a_{d-1}\) so that, on the \(j\)th forward move, the leading digit \(a_j\) permits the appended digit \(z_j\). The intermediate leading digits are \(a_1,\ldots,a_{d-1},z_0\), all nonzero. Therefore every intermediate block is active and can be filled before the next move. After \(d\) moves the target block is filled. Hence \(\mathcal O\) contains every point in every nonzero pile. Finally, let \(0<x<n\). Write
\[
  x=k^j y,\qquad k\nmid y.
\]
By \eqref{eq:inverse-shuffle}, \(\sigma^{-1}(y)\) lies in the nonzero pile \(y\bmod k\), so it belongs to \(\mathcal O\). Thus \(y\in\mathcal O\). While a point \(z<n\) lies in pile \(0\), \eqref{eq:shuffle} gives \(\sigma(z)=kz\), and hence
\[
  y,ky,\ldots,k^j y=x\in\mathcal O.
\]
Therefore
\[
  1^L=\Omega_{k,n}\setminus\{0\},
\]
so \(L\), and hence \((H_{k,n})_0\), is transitive on the points different from \(0\). Moreover,
\[
  \rho(0)=n\neq0.
\]
For any \(y\neq0\), choose \(\ell\in L\) with \(\ell(n)=y\). Then \(\ell\rho\) sends \(0\) to \(y\), so \(H_{k,n}\) is transitive on \(\Omega_{k,n}\).

Finally, let \((x_1,x_2)\) and \((y_1,y_2)\) be ordered pairs of distinct points. Choose \(g,h\in H_{k,n}\) with
\[
  g(x_1)=0,\qquad h(y_1)=0.
\]
Then \(g(x_2)\) and \(h(y_2)\) are both nonzero. Since \((H_{k,n})_0\) is transitive on the nonzero points, there is an element \(v\in(H_{k,n})_0\) such that
\[
  v(g(x_2))=h(y_2).
\]
The element \(h^{-1}vg\) sends \(x_1\) to \(y_1\) and \(x_2\) to \(y_2\). Hence \(H_{k,n}\) is \(2\)-transitive.
\end{proof}

\section{The residual cycle case}
\label{sec:residual-cycle}

\begin{proof}[Proof of Theorem~\ref{thm:residual-cycle}]
Assume \(\lambda\geq 3\). The support graph \(C(k,r)\) from \eqref{eq:support-graph} is a union of cycles of length at least \(3\). Every nonzero vertex is incident with an edge that avoids \(0\). Indeed, this is automatic on cycles not containing \(0\), while deleting \(0\) from its own cycle leaves a path through all remaining vertices of that cycle. Each edge of \(C(k,r)\) is the support of a conjugate of \(c\), and on its two incident piles that conjugate translates the residual coordinate by \(1\) and \(-1\). If the edge avoids \(0\), the conjugate fixes \(0\). Therefore every nonzero pile has a residual unit gate in \((H_{k,n})_0\). Proposition~\ref{prop:full-gate} proves that \(H_{k,n}\) is \(2\)-transitive.
\end{proof}

\section{The antipodal support case}
\label{sec:antipodal}

We now assume
\begin{equation}
  k=2h,\qquad \gcd(k,t)=h.
\label{eq:antipodal-hypothesis}
\end{equation}
Then
\begin{equation}
  t=hu
\label{eq:antipodal-residual}
\end{equation}
for an odd positive integer \(u\). Write \(u=2m+1\). In Lemma~\ref{lem:shift-register}, the remainder is \(r=h\), and
\[
  \varepsilon_a=
  \begin{cases}
  0,&0\leq a<h,\\
  1,&h\leq a<2h.
  \end{cases}
\]

Define
\begin{equation}
  \eta=\rho^{-h}u_d,\qquad
  \eta^*=\rho^h\eta\rho^{-h}.
\label{eq:eta-definitions}
\end{equation}
Both elements preserve every pile. Directly from \eqref{eq:ud-action},
\[
\eta(a,B,X)=
\begin{cases}
(a,B,X+m),&0\leq a<h,\\
(a,B,X+m+1),&h\leq a<2h,
\end{cases}
\]
\[
\eta^*(a,B,X)=
\begin{cases}
(a,B,X+m+1),&0\leq a<h,\\
(a,B,X+m),&h\leq a<2h.
\end{cases}
\]
Consequently the element
\begin{equation}
  \theta=\eta^{m+1}(\eta^*)^{-m}
\label{eq:theta-definition}
\end{equation}
acts as
\begin{equation}
\theta(a,B,X)=
\begin{cases}
(a,B,X),&0\leq a<h,\\
(a,B,X+u),&h\leq a<2h.
\end{cases}
\label{eq:theta-action}
\end{equation}
In particular, \(\theta\in(H_{k,n})_0\).

The conjugates of \(c\) now have supports
\[
  \{a,a+h\},\qquad a\in\mathbb Z/2h\mathbb Z.
\]
All such antipodal pairs except \(\{0,h\}\) avoid \(0\). Thus every pile label other than \(0\) and \(h\) has a residual unit gate in the point stabilizer.

\begin{proposition}[]\label{prop:antipodal}
Assume \eqref{eq:antipodal-hypothesis} and \(h\geq 2\). Then \(H_{k,n}\) is \(2\)-transitive.
\end{proposition}

\begin{proof}
If \(u=1\), then \eqref{eq:theta-action} is a residual unit translation on every pile in the second half, including pile \(h\). Together with the safe conjugates of \(c\), this supplies a unit gate on every nonzero pile. Proposition~\ref{prop:full-gate} applies.

Assume henceforth that \(u\geq 3\). Then
\begin{equation}
  t=hu\geq 3h>2h=k.
\label{eq:antipodal-large-residual}
\end{equation}
Let
\[
  S=\{1,2,\ldots,2h-1\}\setminus\{h\}.
\]
Call a block \emph{safe} if its leading digit belongs to \(S\). Every safe block has a residual unit gate arising from a conjugate of \(c\) whose support avoids \(0\). Let \(L\leq(H_{k,n})_0\) be generated by \(\sigma\), all those safe conjugates of \(c\), and \(\theta\), and put \(\mathcal O=1^L\). As before, \(\mathcal O\) fills the seed block \(10^{d-1}\). By \(d-1\) backward moves with residual choice \(X=1\), it fills \(1^d\). We claim that every safe block is filled. Let its word be \(z_0z_1\cdots z_{d-1}\), with \(z_0\in S\). Starting at the filled block \(1^d\), apply \(d\) forward moves, successively appending \(z_0,z_1,\ldots,z_{d-1}\). To append a prescribed quotient \(q\), we need an integer \(kX+1\) in \([qt,(q+1)t-1]\). By \eqref{eq:antipodal-large-residual}, every interval of \(t\) consecutive integers contains a representative of the residue class \(1\pmod k\), so such an \(X\in[0,t-1]\) exists. During the first \(d-1\) moves the leading digit is \(1\), and after the last move it is \(z_0\). Hence every intermediate block is safe and can be filled before the next move. This proves the claim. Suppose first that \(h\geq 3\). Fix a lower digit word \(B\) and a residue class \(x\pmod u\), with \(0\leq x<u\). The orbit of \((h,B,x)\) under \(\theta\) consists of
\begin{equation}
  (h,B,x+ju),\qquad 0\leq j<h.
\label{eq:theta-residue-orbit}
\end{equation}
Write \(B=(b_1,\ldots,b_{d-1})\). On the first application of \(\sigma\), the base \(k\) word changes from
\[
  hb_1\cdots b_{d-1}
\]
to
\[
  b_1\cdots b_{d-1}q_j,
\]
where the appended quotient is
\begin{equation}
\begin{aligned}
q_j&=\left\lfloor\frac{k(x+ju)+h}{t}\right\rfloor\\
&=\left\lfloor\frac{2x+1}{u}\right\rfloor+2j.
\end{aligned}
\label{eq:antipodal-quotient}
\end{equation}
Each subsequent application of \(\sigma\) removes the current leading digit, shifts every remaining digit one place to the left, and appends a new quotient. Consequently, after exactly \(d\) applications, the first appended quotient \(q_j\) is the leading digit. The initial term in \eqref{eq:antipodal-quotient} is either \(0\) or \(1\), and the \(h\) values \(q_0,\ldots,q_{h-1}\) are distinct elements of \(\{0,1,\ldots,2h-1\}\). Since \(h\geq 3\), they cannot all belong to the two-element set \(\{0,h\}\). Choose \(j\) with \(q_j\in S\). Then \(\sigma^d(h,B,x+ju)\) lies in a safe block, hence belongs to \(\mathcal O\). Applying \(\sigma^{-d}\) shows that \((h,B,x+ju)\in\mathcal O\), and powers of \(\theta\) fill the entire class \eqref{eq:theta-residue-orbit}. Varying \(x\) fills every block in pile \(h\). It remains to treat \(h=2\), so \(k=4\), \(t=2u\), and \(S=\{1,3\}\). Let \(z_0z_1\cdots z_{d-1}\) be a block word with \(z_0=2\), and let \(Y\in\{0,1,\ldots,t-1\}\) be odd. Put \(q=z_{d-1}\) and \(M=qt+Y\). Since \(M\) is odd, there is a unique \(a\in\{1,3\}\) with \(M\equiv a\pmod 4\). Set
\begin{equation}
  X=\frac{M-a}{4}.
\end{equation}
Since \(M\) is a positive odd integer, the selected \(a\) is its least positive residue modulo \(4\), and hence \(a\leq M\). Moreover, \(q\leq3\) and \(Y\leq t-1\), so
\[
  M=qt+Y\leq4t-1.
\]
Consequently,
\[
  0\leq M-a<4t,
\]
and therefore \(0\leq X<t\). In digit residual notation, and with the evident interpretation when \(d=1\), the shift register formula gives
\begin{equation}
\sigma(a,z_0,z_1,\ldots,z_{d-2},X)
=(z_0,z_1,\ldots,z_{d-1},Y).
\end{equation}
The source block is safe, so the target point belongs to \(\mathcal O\). Thus every point of pile \(2\) with odd residual coordinate is in \(\mathcal O\). On pile \(2\), the element \(\theta\) adds \(u\) to the residual coordinate. Since \(u\) is odd and \(t=2u\) is even, this operation interchanges residual parity. Hence every point in pile \(2\) belongs to \(\mathcal O\). In either case, \(\mathcal O\) contains every point in every nonzero pile. The final pile \(0\) argument in the proof of Proposition~\ref{prop:full-gate} applies verbatim, and \(H_{k,n}\) is \(2\)-transitive.
\end{proof}

\begin{proof}[Proof of Theorem~\ref{thm:two-transitive}]
Write \(n=k^s t\) with \(t>1\). Since \(k\nmid t\), the integer \(\lambda=k/\gcd(k,t)\) is at least \(2\). If \(\lambda\geq 3\), apply Theorem~\ref{thm:residual-cycle}. If \(\lambda=2\), then \(k\) is even; as \(k\geq 3\), write \(k=2h\) with \(h\geq 2\). We have \(\gcd(k,t)=h\), so \(t=hu\) with \(u\) odd, and Proposition~\ref{prop:antipodal} applies. Therefore \(H_{k,n}\) is always \(2\)-transitive.
\end{proof}

\section{From two-transitivity to the alternating group}
\label{sec:alternating-containment}

The controlled commutator also provides a large fixed point ratio element.

\begin{lemma}\label{lem:section-seven-fpr}
The element \(c=[\rho,u_d]\) has order \(t\) and
\begin{equation}
  \fpr(c)=\frac{k-2}{k}.
\label{eq:commutator-fpr}
\end{equation}
For every prime \(p\mid t\), the element
\begin{equation}
  x=c^{t/p}
\label{eq:prime-order-power}
\end{equation}
has order \(p\) and the same fixed point ratio.
\end{lemma}

\begin{proof}
By Proposition~\ref{prop:gate}, \(c\) fixes \(k-2\) piles pointwise. On each of the other two piles, and for each of the \(k^s\) choices of the lower word \(B\), it acts on \(X\in\mathbb Z/t\mathbb Z\) as translation by \(1\) or \(-1\). Thus it has order \(t\) and fixes exactly \((k-2)n\) points. Raising a \(t\)-cycle to the power \(t/p\) produces cycles of length \(p\) and no fixed points.
\end{proof}

We use two consequences of the fixed point ratio classification of Burness and Guralnick.

\begin{theorem}[]
\label{thm:bg-product-action}
Let \(G\) be a finite primitive permutation group and let \(x\in G\) have prime order \(p\). Then either
\begin{equation}
  \fpr(x)\leq (p+1)^{-1/2},
\label{eq:bg-product-bound}
\end{equation}
or, up to permutation isomorphism, there are integers \(M,e,\ell\) with
\(e\geq 1\) and \(1\leq\ell<M/2\) such that
\begin{equation}
  \Alt(M)^e\leq G\leq\Sym(M)\wr\Sym(e)
\label{eq:bg-product-action}
\end{equation}
in product action on \(\binom{[M]}{\ell}^{\,e}\).
\end{theorem}

Here \([M]=\{1,\ldots,M\}\). This is \cite[Corollary~2 and Remark~2, p.~5]{BG}, Corollary~2 gives the dichotomy, and Remark~2 supplies the normalization \(1\leq\ell<M/2\). Its hypotheses are primitivity and prime order, with no restriction to odd primes, so they include \(p=2\). The threshold is exactly \((p+1)^{-1/2}\); the bound \(p^{-1}\) belongs instead to the almost simple Corollary~3. At \(p=2\),
\[
  (p+1)^{-1/2}=3^{-1/2}\approx 0.57735<\frac35.
\]
Thus the power of two residual family is already covered when \(k=5\).

\begin{theorem}[]
\label{thm:bg-almost-simple-action}
Let \(G\) be a finite almost simple primitive permutation group and let \(x\in G\) have prime order \(p\). Then either
\[
  \fpr(x)\leq\frac1p,
\]
or, up to permutation isomorphism, \(G\) is an alternating or symmetric group on \(\ell\) subsets, or \((G,x)\) is one of the classical subspace action exceptions in Table 1 of \cite{BG}.
\end{theorem}

This is \cite[Corollary~3, p.~5, and Table~1, p.~6]{BG}.

\begin{corollary}[]
\label{cor:order-three-specialization}
Let \(G\) be a finite almost simple primitive permutation group and let \(x\in G\) have order \(3\) with \(\fpr(x)>1/3\). Then either
\begin{enumerate}[(i)]
\item \(G=\Sym(M)\) or \(\Alt(M)\) acting on \(\ell\) subsets of \([M]\), where \(1\leq\ell<M/2\); or
\item \(G\) is the exceptional classical action with
\[
  (G_0,H,p,\fpr(x))
  =\bigl(\operatorname{Sp}_6(2),\operatorname{O}_6^-(2),3,5/14\bigr)
\]
from Table~1 of \cite{BG}.
\end{enumerate}
\end{corollary}

\begin{proof}
Inspect \cite[Table~1, p.~6]{BG}.  Its unique order \(3\) row outside the subset actions is the stated \(\operatorname{Sp}_6(2)\) action.
\end{proof}

\begin{lemma}[]
\label{lem:collapse-subset-product}
Suppose a group \(G\) is \(2\)-transitive on its domain \(\Omega\) and occurs in the action \eqref{eq:bg-product-action}. Then \(e=\ell=1\), and consequently \(\Alt(\Omega)\leq G\).
\end{lemma}

\begin{proof}
If \(e\geq2\), then the induced product action preserves Hamming distance on ordered pairs of distinct points. There are pairs at Hamming distance one and pairs at Hamming distance two, so no subgroup of the wreath product is \(2\)-transitive. Hence \(e=1\). If \(\ell\geq2\), then the induced \(\Sym(M)\) action on the set of \(\ell\)-subsets preserves \(|A\cap B|\) on ordered pairs \((A,B)\). Since \(\ell<M/2\), there are two disjoint \(\ell\) subsets, and there are two distinct \(\ell\)-subsets meeting in \(\ell-1\) points. Thus there are at least two orbits on ordered pairs of distinct points, again contradicting \(2\)-transitivity. Therefore \(\ell=1\), and \eqref{eq:bg-product-action} is the natural action.
\end{proof}

\begin{proposition}[]
\label{prop:alternating-at-least-five}
Let \(k\geq 5\), let \(n=k^s t\) with \(t>1\), and suppose that \(H_{k,n}\) is \(2\)-transitive. Then
\[
  \Alt(kn)\leq H_{k,n}.
\]
\end{proposition}

\begin{proof}
Choose a prime \(p\mid t\) and let \(x=c^{t/p}\). By Lemma~\ref{lem:section-seven-fpr},
\[
  \fpr(x)=\frac{k-2}{k}\geq\frac35>3^{-1/2}
  \geq(p+1)^{-1/2}.
\]
The strict endpoint inequality is exact: after squaring positive quantities, \(3/5>1/\sqrt3\) is equivalent to \(27>25\). Thus the bounded alternative in \cite[Corollary~2, p.~5]{BG} is impossible, including in the limiting case \(k=5\), \(p=2\). The group \(H_{k,n}\) is primitive because it is \(2\)-transitive. Theorem~\ref{thm:bg-product-action} and Lemma~\ref{lem:collapse-subset-product} give the result.
\end{proof}

\begin{proposition}[]
\label{prop:alternating-four-piles}
Let \(n=4^s t\) with \(t>1\) and \(4\nmid t\).
\begin{enumerate}[(i)]
\item If \(t=2\), then \(H_{4,n}\cong\AGL(2s+3,2)\).
\item If \(t\neq 2\) and \(H_{4,n}\) is \(2\)-transitive, then \(\Alt(4n)\leq H_{4,n}\).
\end{enumerate}
\end{proposition}

\begin{proof}
Part (i) is Proposition~\ref{prop:four-pile-affine}, since \(n=2\cdot4^s\). For part (ii), first suppose that \(t\) has a prime divisor \(p\geq 5\). The element \(x=c^{t/p}\) has order \(p\) and fixed-point ratio \(1/2\), so
\[
  \fpr(x)=\frac12>(p+1)^{-1/2}.
\]
Theorem~\ref{thm:bg-product-action} and Lemma~\ref{lem:collapse-subset-product} give \(\Alt(4n)\leq H_{4,n}\). It remains to assume that every prime divisor of \(t\) belongs to \(\{2,3\}\). Since \(4\nmid t\), we have
\[
  t=3^a\qquad\text{or}\qquad t=2\cdot 3^a
\]
for some \(a\geq 0\). The cases \(a=0\) are \(t=1\), which is excluded, and \(t=2\), which is part (i). Hence \(a\geq 1\). The element \(x=c^{t/3}\) has order \(3\) and
\begin{equation}
  \fpr(x)=\frac12.
\label{eq:four-pile-order-three-fpr}
\end{equation}

A finite \(2\)-transitive group is either of affine type or almost simple; see, for example, \cite[Theorem~4.1B]{DM}. Here the degree
\[
  4n=4^{s+1}t
\]
is divisible by both \(2\) and \(3\), and therefore is not a prime power. Thus \(H_{4,n}\) is almost simple. Apply Corollary~\ref{cor:order-three-specialization}. Its exceptional \(\operatorname{Sp}_6(2)\) action has fixed point ratio \(5/14\), not \(1/2\), so it cannot occur. We are left with an alternating or symmetric group on \(\ell\) subsets. Two-transitivity forces \(\ell=1\) by the intersection size argument in Lemma~\ref{lem:collapse-subset-product}. Therefore \(\Alt(4n)\leq H_{4,n}\).
\end{proof}

\section{Completion of the three pile case}
\label{sec:three-pile}

Set \(H_n=H_{3,n}\) and \(N=3n\). If \(n=3^s t\) with \(3\nmid t\), the controlled commutator \(c\) from Proposition~\ref{prop:gate} has fixed point ratio \(1/3\). Before the classified large-degree step, Corollary 3 of Burness--Guralnick settles every residual factor with an odd prime divisor.

\begin{proposition}[]
\label{prop:odd-prime-residuals-three}
Let \(n=3^s t\) with \(t>1\) and \(3\nmid t\). If \(t\) is not a power of \(2\), then
\[
  \Alt(3n)\leq H_n.
\]
\end{proposition}

\begin{proof}
Theorem~\ref{thm:two-transitive} makes \(H_n\) \(2\)-transitive. Choose an odd prime \(p\mid t\), necessarily \(p\geq 5\). By Lemma~\ref{lem:section-seven-fpr}, the element \(x=c^{t/p}\) has order \(p\) and \(\fpr(x)=1/3\). If \(p\geq 11\), then
\[
  \frac13>\frac1{\sqrt{p+1}},
\]
so Theorem~\ref{thm:bg-product-action} and Lemma~\ref{lem:collapse-subset-product} give \(\Alt(3n)\leq H_n\). It remains to consider \(p\in\{5,7\}\). The degree \(3n=3^{s+1}t\) has at least two distinct prime divisors, so it is not a prime power. A finite \(2\)-transitive group is affine or almost simple, hence \(H_n\) is almost simple. Since \(1/3>1/p\), apply Theorem~\ref{thm:bg-almost-simple-action}. The orders displayed in \cite[Table~1, p.~6]{BG} show that there is no order \(5\) row. Its first row is the only possible order \(7\) row. There \(q=8\), the socle is \(\PSL(2,8)\), and the displayed formula gives
\[
  \frac17+\frac5{63}=\frac29,
\]
not \(1/3\). Thus only an alternating or symmetric subset action remains, and Lemma~\ref{lem:collapse-subset-product} forces the natural action. Therefore \(\Alt(3n)\leq H_n\).
\end{proof}

The residual power of two family requires a sharper argument. We isolate the large degree classification input, following the \(k=3\) case analysis in \cite[proof of Theorem~1.3]{XZZZ}.

\begin{lemma}[]
\label{lem:liebeck-saxl-consequence}
Let \(G\leq\Sym(N)\) be an almost simple \(2\)-transitive group of degree \(N>276\), and suppose that its socle is a simple group of Lie type over \(\mathbb F_q\), where \(q>4\). If \(x\in G\setminus\{1\}\) satisfies
\[
  \fpr(x)>\frac4{3q},
\]
then \(\Soc(G)=\PSL(2,q)\), the action is the natural action on \(\operatorname{PG}(1,q)\), and \(N=q+1\). Moreover,
\[
  |\Fix(x)|=2
  \qquad\text{or}\qquad
  |\Fix(x)|=q_0+1,
\]
where in the second case \(q=q_0^r\) for a prime power \(q_0\) and an integer \(r\geq2\). Equivalently,
\[
  \fpr(x)=\frac2{q+1}
  \qquad\text{or}\qquad
  \fpr(x)=\frac{q_0+1}{q+1}.
\]
\end{lemma}

\begin{proof}
Liebeck and Saxl \cite[Theorem~1]{LS} classify the exceptions to the bound \(\fpr(x)\leq4/(3q)\) for primitive almost simple groups with socle of Lie type over \(\mathbb F_q\). Intersecting their exceptional list with the almost simple \(2\)-transitive actions in \cite[Table~7.4]{Cameron}, and using \(N>276\), leaves only the natural projective line action with socle \(\PSL(2,q)\). In particular, the higher dimensional projective actions and the unitary, Suzuki and Ree \(2\)-transitive families
\[
  \PSU(3,q),\qquad \Sz(q),\qquad {}^2G_2(q)
\]
lie on the bounded side. In the projective line action, the exceptional fixed point sets in \cite[Theorem~1]{LS} have size either \(2\), or \(q_0+1\) for a proper subfield \(\mathbb F_{q_0}\subsetneq\mathbb F_q\), where \(q=q_0^r\) and \(r\geq2\). Dividing by \(|\operatorname{PG}(1,q)|=q+1\) gives the two displayed fixed point ratios. This is the precise specialization used in \cite[proof of Theorem~1.3, Case~1, p.~6]{XZZZ}.
\end{proof}

\begin{proposition}[]
\label{prop:large-degree-fpr-reduction}
Let \(G\leq\Sym(N)\) be \(2\)-transitive of degree \(N>276\). Assume that
\begin{enumerate}[(a)]
\item \(G\) is not of affine type;
\item \(G\) contains a nonidentity element \(g\) with \(\fpr(g)=1/3\);
\item if \(N\) is even, then \(G\) contains an involution \(h\) with \(\fpr(h)=1/3\).
\end{enumerate}
Then \(\Alt(N)\leq G\).
\end{proposition}

\begin{proof}
The affine/almost simple dichotomy makes \(G\) almost simple. Suppose that \(\Alt(N)\nleq G\). We first explain the cutoff. In the classification of almost simple \(2\)-transitive groups, the alternating socle target in unbounded degree is the natural action and already contains \(\Alt(N)\). The sporadic socles are
\[
  M_{11},\ M_{12},\ M_{22},\ M_{23},\ M_{24},\ HS,\ Co_3,
\]
and their \(2\)-transitive degrees are at most \(276\), with equality for the \(Co_3\) action. Thus Cameron's table \cite[Table~7.4]{Cameron}, after the alternating target and sporadic rows are discarded, leaves a simple socle of Lie type over some field \(\mathbb F_q\).

\textbf{Case 1: \(q>4\).}
Since \(\fpr(g)=1/3>4/(3q)\), Lemma~\ref{lem:liebeck-saxl-consequence} gives
\(\Soc(G)=\PSL(2,q)\) and
\[
  \frac13=\frac2{q+1}
  \qquad\text{or}\qquad
  \frac13=\frac{q_0+1}{q+1},
  \qquad q=q_0^r,\quad r\geq 2.
\]
The first equality gives \(q=5\) and \(N=6\). The second is \(q_0^r=3q_0+2\). If \(q_0\geq 4\), then \(q_0^r\geq q_0^2>3q_0+2\); \(q_0=3\) gives \(11\), not a power of \(3\); and \(q_0=2\) gives the unique solution \(r=3\), so \(q=8\) and \(N=9\). Both contradict \(N>276\).

\textbf{Case 2: \(q=4\).}
The \(2\)-transitive list leaves the natural actions of \(\PGammaU(3,4)\) and \(\PGammaL(d,4)\). The unitary degree is \(4^3+1=65\). In the projective case,
\[
  N=\frac{4^d-1}{3}>276
\]
forces \(d\geq 5\). By Guralnick--Kantor \cite[Proposition~3.1]{GK}, every prime order nonidentity element has fixed point ratio strictly less than
\[
  \min\left\{\frac12,\frac14+\frac1{4^{d-1}}\right\}
  \leq\frac14+\frac1{4^4}
  =\frac{65}{256}<\frac13,
\]
a contradiction. Indeed, a suitable prime order power \(y\) of \(g\) is nonidentity and satisfies \(\Fix(g)\subseteq\Fix(y)\), hence \(\fpr(y)\geq1/3\).

\textbf{Case 3: \(q=3\).}
Here \(G\leq\PGL(d,3)\) on the projective points of \(\mathbb F_3^d\). Choose a linear representative \(T\in\GL(d,3)\) of \(g\). If \(u_\lambda\) is the dimension of the eigenspace associated to a rational eigenvalue \(\lambda\in\mathbb F_3^\times\), then
\[
  |\Fix(g)|=\sum_{\lambda\in\mathbb F_3^\times}
  \frac{3^{u_\lambda}-1}{2},
\]
and hence
\[
  \fpr(g)=
  \frac{\displaystyle\sum_{\lambda\in\mathbb F_3^\times}
  (3^{u_\lambda}-1)}
  {3^d-1}.
\]
Equality with \(1/3\) would imply
\[
  3\sum_{\lambda\in\mathbb F_3^\times}(3^{u_\lambda}-1)=3^d-1,
\]
which is impossible modulo \(3\).

\textbf{Case 4: \(q=2\).}
The list leaves \(\PSL(d,2)\) on the nonzero vectors and the two symplectic actions of \(\operatorname{Sp}(2m,2)\). For \(\PSL(d,2)=\GL(d,2)\), the fixed vectors of \(g\) form a subspace, so \(|\Fix(g)|=2^r-1\) for some \(r\leq d\). Thus
\[
  \frac13=\frac{2^r-1}{2^d-1},\qquad
  3\cdot 2^r=2^d+2.
\]
The right side is \(2\) modulo \(4\), forcing \(r=1\), and then \(2^d=4\), which is incompatible with \(N>276\). For the symplectic rows, Cameron's table has \(G=G_0\), where \(G_0=\operatorname{Sp}(2m,2)\). The two degrees are
\[
  N=2^{m-1}(2^m+\varepsilon),\qquad
  \varepsilon\in\{+1,-1\}.
\]
They are even, so hypothesis (c) supplies an involution \(h\) with \(\fpr(h)=1/3\). Put
\[
  H=\operatorname{O}^\varepsilon(2m,2),
  \qquad [G_0:H]=2^{m-1}(2^m+\varepsilon).
\]
Since \(\fpr(h)>0\), the involution fixes a coset. After conjugating it, we may assume \(h\in H\). The fusion statement of Aschbacher--Seitz \cite[(8.5)]{AS} gives \(h^{G_0}\cap H=h^H\), and hence
\[
  \fpr(h)=
  \frac{|H|\,|C_{G_0}(h)|}{|G_0|\,|C_H(h)|}.
\]
For a finite group \(K\), let \(O_2(K)\) denote its largest normal \(2\)-subgroup. Lemma~2.2 of \cite[pp.~4--5]{XZZZ}, which is derived there from the involution centralizers of Aschbacher--Seitz, gives an integer \(r\) with \(1\leq r\leq m\) such that
\[
  \frac{|C_{G_0}(h)|}{|C_H(h)|}
  =
  \frac{|\operatorname{Sp}(2m-2r,2)|}
       {|\operatorname{O}^{\varepsilon}(2m-2r,2)|}
  \frac{|O_2(C_{G_0}(h))|}{|O_2(C_H(h))|}.
\]
Both groups in the last quotient are \(2\)-groups. Define
\[
  \delta=
  v_2\bigl(|O_2(C_{G_0}(h))|\bigr)
  -v_2\bigl(|O_2(C_H(h))|\bigr),
\]
so that the last quotient is \(2^\delta\). This makes explicit the power of two suppressed in the calculation in \cite[proof of Theorem~1.3, Case~4, pp.~6--7]{XZZZ}.

If \(r=m\), then the symplectic and orthogonal groups of dimension zero are trivial, and hence
\[
  \frac13=\fpr(h)=
  \frac{2^{\delta-m+1}}{2^m+\varepsilon}.
\]
The denominator \(2^m+\varepsilon\) is odd. Taking \(2\)-adic valuations gives \(\delta-m+1=0\), after which the equality gives \(2^m+\varepsilon=3\). This is impossible for \(m\geq3\).

Suppose that \(r<m\), and put \(a=m-r\). The standard order formulas give
\[
  \frac{|\operatorname{Sp}(2a,2)|}
       {|\operatorname{O}^{\varepsilon}(2a,2)|}
  =2^{a-1}(2^a+\varepsilon).
\]
Consequently
\[
  \frac13=\fpr(h)=
  2^{\delta-r}
  \frac{2^{m-r}+\varepsilon}{2^m+\varepsilon}.
\]
The fraction on the right has odd numerator and odd denominator, so comparison of \(2\)-adic valuations gives \(\delta=r\). Therefore
\[
  2^m=3\cdot2^{m-r}+2\varepsilon.
\]
If \(m-r\geq2\), the right hand side is congruent to \(2\pmod4\), whereas the left hand side is divisible by \(4\). Hence \(m-r=1\). The equation then becomes \(2^m=6+2\varepsilon\), whose only solution with \(m\geq3\) is \((m,\varepsilon)=(3,+1)\). The corresponding degree is
\[
  2^{m-1}(2^m+\varepsilon)=36,
\]
contrary to \(N>276\).

All Lie type cases are impossible, so \(\Alt(N)\leq G\).
\end{proof}

The group \(\PGammaL(2,8)\) is \(2\)-transitive on the nine points of \(\operatorname{PG}(1,8)\). Its order \(3\) field automorphism fixes the subline \(\operatorname{PG}(1,2)\), hence has fixed point ratio \(1/3\), while the group does not contain \(\Alt(9)\). This is exactly the \(q=8\) solution in Case 1. Thus the large degree hypothesis is substantive rather than cosmetic: the value \(276\) removes every sporadic \(2\)-transitive row, with the largest degree supplied by \(Co_3\).

\begin{corollary}[]\label{cor:large-three-pile-range}
If \(n>92\) is not a power of \(3\), then \(\Alt(3n)\leq H_n\).
\end{corollary}

\begin{proof}
Theorem~\ref{thm:two-transitive} gives \(2\)-transitivity. Since \(n=3^s t\) with \(t>1\) and \(3\nmid t\), the degree \(3^{s+1}t\) is not a prime power, so \(H_n\) is not affine. Lemma~\ref{lem:section-seven-fpr} supplies an element of fixed point ratio \(1/3\). If the degree is even, then \(t\) is even and \(c^{t/2}\) is an involution with the same fixed point ratio. Apply Proposition~\ref{prop:large-degree-fpr-reduction}.
\end{proof}

Only the residual power of two values remain after Proposition~\ref{prop:odd-prime-residuals-three}. A second pass through the complete \(2\)-transitive list reduces the finite boundary before any explicit cycle calculation.

\begin{proposition}[]
\label{prop:finite-three-pile-reduction}
Let
\[
  n=3^s2^j\leq92,
  \qquad s\geq0,\quad j\geq1.
\]
If \(\Alt(3n)\not\leq H_n\), then
\[
  n\in\{2,4,8,12\}.
\]
\end{proposition}

\begin{proof}
Put \(N=3n\). Theorem~\ref{thm:two-transitive} makes \(H_n\) \(2\)-transitive, and Lemma~\ref{lem:section-seven-fpr} supplies the involution
\[
  h=c^{\,2^{j-1}},
  \qquad \fpr(h)=\frac13.
\]
The possible degrees are
\[
\begin{split}
\mathcal D=\{&6,12,18,24,36,48,54,72,96,108,\\
             &144,162,192,216\}.
\end{split}
\]
Every member of \(\mathcal D\) is divisible by both \(2\) and \(3\), so it is not a prime power. Hence \(H_n\) is not affine and is therefore almost simple. Suppose that \(\Alt(N)\not\leq H_n\). We now inspect every nonaffine row of the complete \(2\)-transitive list in \cite[Table~7.4]{Cameron}. The alternating group in its natural action already contains \(\Alt(N)\), so that row is excluded by assumption.

\emph{Projective rows.}
Suppose
\[
  \PSL(d,q)\leq H_n\leq\PGammaL(d,q),
  \qquad
  N=\frac{q^d-1}{q-1}=1+q+\cdots+q^{d-1}.
\]
If \(q\) is even, then \(N\) is odd. If \(q\) is odd, then \(N\equiv d\pmod2\). Since \(N\) is even, \(q\) is odd and \(d\) is even. If \(d\geq6\), then
\[
  N\geq1+3+3^2+3^3+3^4+3^5=364,
\]
which is impossible. If \(d=4\), the bound \(N\leq216\) gives \(q\in\{3,5\}\), and the corresponding degrees are \(40\) and \(156\), neither of which belongs to \(\mathcal D\).

It remains to take \(d=2\), so \(N=q+1\). If \(h\in\PGL(2,q)\), then a nonidentity projective linear transformation fixes at most two points of \(\operatorname{PG}(1,q)\). Since \(|\Fix(h)|=(q+1)/3\), this leaves only \(q+1=6\). Suppose that \(h\notin\PGL(2,q)\). The field component of the involution \(h\) has order two, so \(q=q_0^2\) and its fixed field is \(\mathbb F_{q_0}\). Write \(h=A\varphi\), where \(\varphi\) is the field involution. Since \(h^2=1\) projectively, one has \(AA^\varphi=\lambda I\) for some \(\lambda\in\mathbb F_q^\times\). Applying \(\varphi\) and using \(A^\varphi A=\lambda I\) shows that \(\lambda^\varphi=\lambda\), so \(\lambda\in\mathbb F_{q_0}^{\times}\). Surjectivity of the norm \(\mathbb F_q^{\times}\to\mathbb F_{q_0}^{\times}\) permits a scalar rescaling of \(A\) such that \(AA^\varphi=I\). The matrix form of Hilbert's Theorem~90 then gives \(A=B^{-1}B^\varphi\), so \(h\) is projectively conjugate to \(\varphi\). It therefore fixes exactly the subline \(\operatorname{PG}(1,q_0)\). The equality \(\fpr(h)=1/3\) would give
\[
  \frac{q_0+1}{q_0^2+1}=\frac13,
  \qquad
  q_0^2-3q_0-2=0,
\]
which has no integral solution. Thus the projective rows leave only \(N=6\).

\emph{Symplectic rows.}
Here
\[
  N=2^{d-1}(2^d\pm1),
  \qquad d\geq3.
\]
For \(N\leq216\), the possible degrees are
\[
  28,\ 36,\ 120,\ 136;
\]
the next degree is \(496\). Their intersection with \(\mathcal D\) is \(\{36\}\).

\emph{Unitary, Suzuki and Ree rows.}
The simple unitary family has degree \(q^3+1\), with \(q>2\). Up to \(216\), these degrees are
\[
  3^3+1=28,\qquad
  4^3+1=65,\qquad
  5^3+1=126,
\]
none of which belongs to \(\mathcal D\). A Suzuki degree \(q^2+1\) is odd. In the simple Ree family the least parameter is \(q=27\), and \(q^3+1>216\).

\emph{Isolated and sporadic rows.}
The remaining rows of \cite[Table~7.4]{Cameron} have the following groups and degrees:
\[
\begin{array}{c|c}
\text{group or socle}&\text{degree}\\ \hline
\PSL(2,11),\ M_{11}&11\\
M_{11},\ M_{12}&12\\
\Alt(7)&15\\
M_{22}\text{ or }\operatorname{Aut}(M_{22})&22\\
M_{23}&23\\
M_{24}&24\\
\PGammaL(2,8)&28\\
HS&176\\
Co_3&276
\end{array}
\]
Their degrees intersect \(\mathcal D\) in \(\{12,24\}\).

All rows of the almost simple \(2\)-transitive list have now been exhausted. Consequently
\[
  N\in\{6,12,24,36\},
\]
and hence \(n\in\{2,4,8,12\}\).
\end{proof}

\begin{lemma}[]
\label{lem:cycle-isolation-boundary}
Let \(w\in\Sym(N)\) have an orbit of prime length \(p\), where
\[
  \frac N2<p<N.
\]
Then \(w^{(N-p)!}\) is a single \(p\)-cycle.
\end{lemma}

\begin{proof}
The complement of the \(p\)-orbit has \(N-p\) points.  Hence every nontrivial cycle of \(w\) on that complement has length at most \(N-p\), and its length divides \((N-p)!\).  Thus \(w^{(N-p)!}\) fixes the complement pointwise. Since \(N-p<p\) and \(p\) is prime, \(p\nmid(N-p)!\).  Therefore the restriction of \(w^{(N-p)!}\) to the \(p\)-orbit is still a \(p\)-cycle.
\end{proof}

\begin{proposition}[]
\label{prop:explicit-power-two-boundary}
For \(n\in\{2,4,8,12\}\), one has \(\Alt(3n)\leq H_n\).
\end{proposition}

\begin{proof}
We derive every cycle directly from one modular formula.  Put \(N=3n\), and write \(\langle z\rangle_m\) for the least nonnegative residue of \(z\) modulo \(m\).  Formula~\eqref{eq:shuffle} gives
\[
  \sigma(x)=\langle3x\rangle_{N-1}\quad(0\leq x<N-1),
  \qquad \sigma(N-1)=N-1,
\]
while \(\rho(y)=\langle y+n\rangle_N\).  Hence, for \(w_{n,a}=\rho\sigma^a\),
\begin{equation}
  w_{n,a}(x)=\langle r_{n,a}(x)+n\rangle_N,\qquad
  r_{n,a}(x)=
  \begin{cases}
    \langle3^a x\rangle_{N-1},&0\leq x<N-1,\\
    N-1,&x=N-1.
  \end{cases}
\label{eq:boundary-map}
\end{equation}
Thus every arrow below is checked by one reduction modulo \(N-1\), followed by addition of \(n\) modulo \(N\). For \(n=2\), formula~\eqref{eq:boundary-map} gives the complete decomposition
\[
  w_{2,1}=(0\ 2\ 3)(1\ 5),\qquad \Fix(w_{2,1})=\{4\}.
\]
Consequently \(w_{2,1}^3=(1\ 5)\).  Theorem~\ref{thm:two-transitive} and conjugation give every transposition, so \(H_2=\Sym(6)\). For \(n=4\), where \((N-1,3^a,n)=(11,9,4)\), formula~\eqref{eq:boundary-map} gives the orbit
\[
  w_{4,2}:\quad(2\ 11\ 3\ 9\ 8\ 10\ 6).
\]
For \(n=8\), where \((N-1,3^a,n)=(23,3,8)\), it gives the orbit
\[
  w_{8,1}:\quad
  (0\ 8\ 9\ 12\ 21\ 1\ 11\ 18\ 16\ 10\ 15\ 6\ 2\ 14\ 3\ 17\ 13).
\]
Finally, for \(n=12\), where \((N-1,3^a,n)=(35,27,12)\), it gives the orbit
\[
\begin{aligned}
  w_{12,3}:\quad
  &(0\ 12\ 21\ 19\ 35\ 11\ 29\ 25\ 22\ 10\ 1\ 3\\
  &\qquad 23\ 2\ 31\ 8\ 18\ 7\ 26\ 14\ 4\ 15\ 32).
\end{aligned}
\]
In each display the entries are distinct, every entry maps to the next by \eqref{eq:boundary-map}, and the final entry maps back to the first.  Thus these are orbits of the respective prime lengths
\[
  (p_4,p_8,p_{12})=(7,17,23).
\]
Each \(p_n\) satisfies
\begin{equation}
  \frac{3n}{2}<p_n\leq3n-3.
\label{eq:boundary-prime-range}
\end{equation}
By \eqref{eq:boundary-prime-range} and Lemma~\ref{lem:cycle-isolation-boundary}, \(H_n\) contains a single \(p_n\)-cycle. Theorem~\ref{thm:two-transitive} makes \(H_n\) primitive, and \(p_n\leq3n-3\). Jordan's prime cycle theorem \cite[Theorem~3.3E]{DM} therefore gives \(\Alt(3n)\leq H_n\).
\end{proof}

\begin{corollary}[]
\label{cor:alternating-three-piles}
If \(n\) is not a power of \(3\), then
\[
  \Alt(3n)\leq H_{3,n}.
\]
\end{corollary}

\begin{proof}
If \(n>92\), use Corollary~\ref{cor:large-three-pile-range}. Suppose \(n\leq 92\) and write \(n=3^s t\) with \(3\nmid t\). If \(t\) has an odd prime divisor, use Proposition~\ref{prop:odd-prime-residuals-three}. Otherwise \(t\) is a positive power of \(2\). Proposition~\ref{prop:finite-three-pile-reduction} either gives the conclusion directly or reduces to \(n\in\{2,4,8,12\}\), which is Proposition~\ref{prop:explicit-power-two-boundary}.
\end{proof}

\begin{proof}[Proof of Theorem~\ref{thm:classification}]
Part (i) is Proposition~\ref{prop:power-decks}. Part (ii) is Proposition~\ref{prop:four-pile-affine}. Assume that neither exceptional case occurs. Then \(n\) is not a power of \(k\), so Theorem~\ref{thm:two-transitive} gives \(2\)-transitivity. If \(k\geq 5\), Proposition~\ref{prop:alternating-at-least-five} gives \(\Alt(kn)\leq H_{k,n}\). If \(k=4\), write \(n=4^s t\) with \(4\nmid t\), the excluded affine case is exactly \(t=2\), so Proposition~\ref{prop:alternating-four-piles}(ii) applies. If \(k=3\), use Corollary~\ref{cor:alternating-three-piles}. Finally, Proposition~\ref{prop:signs} determines whether the group is alternating or symmetric, proving \eqref{eq:generic-classification}.
\end{proof}

\section{Pile groups containing the standard cycle}
\label{sec:pile-groups}

The cyclic theorem immediately determines a substantially larger family of generalized shuffle groups. For \(P\leq\Sym(k)\), write
\[
  \rho_\tau(an+b)=\tau(a)n+b
  \qquad(\tau\in P),
\]
and set
\[
  J_{P,n}=\Sh(P,n)
  :=\langle\sigma\rho_\tau:\tau\in P\rangle
  =\langle\sigma,\rho_\tau:\tau\in P\rangle.
\]
The equality uses the identity element of \(P\), which supplies \(\sigma\); it also fixes explicitly the shuffle convention used here. If \(C_k\leq P\), then \(H_{k,n}\leq J_{P,n}\).

\begin{theorem}[]\label{thm:cycle-containing-pile-groups}
Let \(k\geq 3\) and let \(C_k\leq P\leq\Sym(k)\).
\begin{enumerate}[(i)]
\item If \(n=k^f\), then
\[
  J_{P,n}\cong P\wr C_{f+1}
\]
in product action.

\item If \(k=4\) and \(n=2\cdot4^j\) for some \(j\geq 0\), then
\[
  J_{P,n}\cong\AGL(2j+3,2).
\]

\item In all remaining cases,
\[
J_{P,n}=
\begin{cases}
\Alt(kn),&
\displaystyle\binom{k}{2}\binom{n}{2}\equiv0\pmod2
\ \text{and}\ 
\bigl(n\equiv0\pmod2\ \text{or}\ P\leq\Alt(k)\bigr),\\
\Sym(kn),&\text{otherwise}.
\end{cases}
\]
\end{enumerate}
\end{theorem}

\begin{proof}
Part (i) is \cite[Theorem~1.4(1)]{AMP}; it also follows directly by identifying the deck with \((f+1)\) base \(k\) digits. The pile group \(P\) acts on one digit, and conjugation by the digit rotation \(\sigma\) supplies independent copies of \(P\) on all coordinates. For part (ii), put \(m=2j+3\), so the deck has size \(2^m\). Identify the pile label with the first two binary coordinates and the depth with the remaining \(m-2\) coordinates. The shuffle \(\sigma\) is the linear cyclic rotation by two coordinates. Every permutation of the four pile labels is affine on \(\mathbb F_2^2\), since \(\AGL(2,2)=\Sym(4)\), and hence every generator \(\rho_\tau\) with \(\tau\in P\) lies in \(\AGL(m,2)\). By Proposition~\ref{prop:four-pile-affine},
\[
  H_{4,n}=\AGL(m,2)\leq J_{P,n},
\]
so equality holds. For part (iii), Theorem~\ref{thm:classification} gives
\(\Alt(kn)\leq H_{k,n}\leq J_{P,n}\), and hence \(J_{P,n}\) is alternating or symmetric. The induced pile permutation \(\rho_\tau\) is the product of \(n\) copies of \(\tau\), so
\[
  \operatorname{sgn}(\rho_\tau)=\operatorname{sgn}(\tau)^n.
\]
Consequently every \(\rho_\tau\) is even precisely when either \(n\) is even or \(P\leq\Alt(k)\). Combining this with \eqref{eq:shuffle-sign} gives the asserted dichotomy.
\end{proof}

\begin{corollary}[]
\label{cor:alternating-pile-groups-odd}
Let \(k\geq 3\) be odd. Then the conjectural classification of \(\Sh(\Alt(k),n)\) in \cite[Conjecture~5.2]{XZZZ} holds for every \(n\):
\[
\Sh(\Alt(k),n)=
\begin{cases}
\Alt(k)\wr C_{f+1},&n=k^f,\\
\Alt(kn),&
n\neq k^f\ \text{and}\ 
\displaystyle\binom{k}{2}\binom{n}{2}\equiv0\pmod2,\\
\Sym(kn),&
n\neq k^f\ \text{and}\ 
\displaystyle\binom{k}{2}\binom{n}{2}\equiv1\pmod2.
\end{cases}
\]
\end{corollary}

\begin{proof}
A \(k\) cycle is even when \(k\) is odd, so \(C_k\leq\Alt(k)\). Apply Theorem~\ref{thm:cycle-containing-pile-groups} with \(P=\Alt(k)\).
\end{proof}

\section{Consequences and scope}
\label{sec:scope}

Theorem~\ref{thm:classification} proves \cite[Conjecture~1.10]{AMP} and all four clauses of \cite[Conjecture~5.1, p.~15]{XZZZ}. Theorem~1.5 of \cite{XZZZ} concerns the full pile group \(\Sh(\Sym(k),n)\), while its Section~5 poses the cyclic group \(\Sh(C_k,n)\) as Conjecture~5.1. Theorem~\ref{thm:cycle-containing-pile-groups} additionally determines every shuffle group whose pile group contains the standard cycle. Corollary~\ref{cor:alternating-pile-groups-odd} therefore proves the odd \(k\) portion of \cite[Conjecture~5.2]{XZZZ}, including \(k=3\), where \(\Alt(3)=C_3\). The even \(k\) portion of Conjecture~5.2 and the subsequent Conjecture~5.3 and Question~5.4 are not claimed. Theorem~\ref{thm:two-transitive} is constructive. Lemma~\ref{lem:digit-connectivity} also gives a uniform
statement about the digit digraph \(D_{k,t}\): it is strongly connected for all \(k,t\geq 2\), has a loop, and has directed diameter at most
\[
  3\lceil\log_t k\rceil-1.
\]
This removes every coprimality assumption from the block propagation mechanism. The threshold in \cite[Corollary~2, p.~5, and proof, p.~80]{BG} is \(1/\sqrt{p+1}\). At \(p=2\) it is \(1/\sqrt3\), and the exact limiting comparison appears in Proposition~\ref{prop:alternating-at-least-five}. The bound \(1/p\) belongs instead to the almost simple Corollary~3.

\section*{Declaration on the Use of AI Tools}
The authors used OpenAI tools to assist with basic language editing of portions of this manuscript. The mathematical content, claims, proofs, and conclusions are the authors' own. The resulting work, in its totality, is an accurate representation of the authors' underlying work, and the authors take full responsibility for the veracity and correctness of all material.

\end{document}